\documentclass{scrartcl}

\usepackage{amsmath}
\usepackage{amsfonts}
\usepackage{amssymb}
\usepackage{amsthm}

\usepackage{graphicx}
\usepackage[english]{babel}

\usepackage[format=plain,labelfont=bf,textfont=it]{caption}

\theoremstyle{definition}
\newtheorem*{defi}{Definition}
\newtheorem{prop}{Proposition}

\renewcommand{\O}{\mathcal{O}}

\title{Modified equations and the Basel problem}
\author{Mats Vermeeren}
\date{\normalsize \textit{Institut f\"ur Mathematik, MA 7-1, Technische Universit\"at Berlin, \\
Str.\@ des 17.\@ Juni 136, 10623 Berlin, Germany \\
E-mail:} \texttt{vermeeren@math.tu-berlin.de}}

\begin{document}

\maketitle

The Basel problem consists in evaluating the series
\[ \sum_{k=1}^\infty \frac{1}{k^2} . \]
It was tackled by many mathematicians in the 17th century, most notably Pietro Mengoli, Gottfried Wilhelm Leibniz and Jakob Bernoulli \cite[p42]{dunham1999euler}. It was finally solved by Leonhard Euler in 1735. He showed that 
\[ \sum_{k=1}^\infty \frac{1}{k^2} = \frac{\pi^2}{6}, \]
a result that still astonishes everyone who sees it for the first time.

Even though Euler's first proof was not rigorous, it convinced contemporary mathematicians. Since then, many proofs of this identity have been found whose rigor stands up to the scrutiny of modern mathematics. A (non-exhaustive) list of such proofs can be found in \cite{chapman1999evaluating}. Some of them are widely praised for their esthetic value, see \cite[Chapter 8]{aigner2010proofs}.

One proof relies on the series expansion
\begin{equation}\label{arcsin-series}
\left( \arcsin \frac{h}{2} \right)^2 = \frac{1}{2} \sum_{k=1}^\infty  \frac{(k-1)!^2}{(2k)!} h^{2k} .
\end{equation} 
Setting $h=1$ in this expansion, combined with some elementary but nontrivial algebraic manipulations, yields the desired result. However, that approach only relocates the difficulty because proving this expansion is not an easy task. Instead of tackling this problem head on, we will turn our attention to a completely different branch of mathematics. 

Modified equations are a powerful tool from numerical analysis. In the next section, we will introduce this concept and start studying an example. As the discussion progresses, it will gradually become clear that the modified equation in this example is intimately related to the series expansion \eqref{arcsin-series}.

The main goal of this article is to show off the concept of modified equations, a beautiful idea which rarely travels beyond the borders of numerical mathematics. However, the point of this work is not to provide a general overview of this subject. Excellent review texts are already available, for example \cite[Chapter IX]{hairer2006geometric}. The point of this article is to show that modified equations have potential use outside of numerical analysis. To illustrate this, at the end of this paper we will use a modified equation to prove Equation \eqref{arcsin-series} and by corollary the Basel problem. 


\section*{Modified equations}

An important technique to study the behavior of a numerical method for an ODE is \emph{backward error analysis}.  Instead of comparing a discrete solution to a solution of the original ODE, one looks for a \emph{modified equation} whose solutions exactly interpolate the numerical solutions. To analyze the behavior of the numerical integrator one then compares the modified equation with the original one. In general a modified equation involves a formal power series, so to extract information in a rigorous way one needs to truncate it. In many cases useful estimates for the truncation error are available, and very strong results can be proved using modified equations. The most famous one is the fact that symplectic integrators for Hamiltonian systems nearly conserve the energy over very long time intervals. 

Let us illustrate the concept of a modified equation in the context of the linear differential equation 
\[\dot{x}(t) = -x(t), \]
where the dot denotes differentiation with respect to $t$. The dependence of $x$ on $t$ will often be suppressed in what follows. We consider the explicit Euler discretization with step size $h$ of this differential equation, which is obtained by replacing the derivative of $x$ by the finite difference $\frac{x_{j+1} - x_j}{h}$ and $x$ itself by $x_j$. For each integer $j$ the quantity $x_j$ is intended to be an approximation of $x(jh)$. This gives us the difference equation 
\[ x_{j+1} - x_j = - h x_j. \]
Its modified equation is a differential equation $\dot{x} = F_h(x)$ whose solutions satisfy the difference equation, in the sense that $x(t+h) - x(t) = - hx(t)$. The right hand side of the modified equation is a function $F:\mathbb{R}_+ \times \mathbb{R} \rightarrow \mathbb{R}: (h,x) \mapsto F_h(x)$, which we assume to have a power series expansion
\[ F_h(x) = f_0(x) + h f_1(x) + h^2 f_2(x) + \ldots. \]
We will assume that this series converges for the step size $h$ under consideration. In most cases the power series defining a modified equation does not converge for any $h$, but in this paper we will only encounter examples where convergence takes place.

The coefficients $f_j(x)$ of the power series can be determined as follows. Using Taylor expansion we can write the condition that $x(t+h) - x(t) = - hx(t)$ as
\[ h \dot{x}(t) + \frac{h^2}{2} \ddot{x}(t) + \frac{h^3}{6} x^{(3)}(t) + \ldots =  - hx(t) . \]
The higher derivatives can be written as
\begin{align*}
\ddot{x} &= F_h'(x)F_h(x) \\
&= \left( f_0'(x) + h f_1'(x) + h^2 f_2'(x) + \ldots \right) \left( f_0(x) + h f_1(x) + h^2 f_2(x) + \ldots \right), \\
x^{(3)} &= F_h''(x)F_h(x)^2 + F_h'(x)^2 F_h(x) \\
&= \left( f_0''(x) + h f_1''(x) + h^2 f_2''(x) + \ldots \right) \left( f_0(x) + h f_1(x) + h^2 f_2(x) +  \ldots \right)^2 \\
&\qquad + \left( f_0'(x) + h f_1'(x) + h^2 f_2'(x) + \ldots \right)^2 \left( f_0(x) + h f_1(x) + h^2 f_2(x) + \ldots \right) ,
\end{align*}
and so on, where the prime denotes differentiation with respect to $x$. Hence the condition becomes
\begin{align*}
& h \left( f_0 + h f_1 + h^2 f_2 + \ldots \right)\\ 
&\qquad + \frac{h^2}{2} \left( f_0'f_0 + h f_1'f_0 + h f_0'f_1 + \ldots \right) + \frac{h^3}{6} \left( f_0''f_0^2 + f_0'^2 f_0 + \ldots \right) + \ldots = -h x,
\end{align*}
where the argument $x$ of the $f_i$ has been omitted. Grouping terms by order in $h$ we find
\[
h(f_0 + x) + h^2 \left( f_1 + \frac{1}{2}f_0'f_0 \right) + h^3 \left( f_2 + \frac{1}{2}f_1'f_0 + \frac{1}{2} f_0'f_1 + \frac{1}{6} f_0''f_0^2 + \frac{1}{6} f_0'^2 f_0 \right) + \ldots = 0 .
\]
For a power series to be equal to zero, all of the coefficients must be zero. From the first order term we find that $f_0(x) = -x$. From the second order term we then learn that $f_1(x) = - \frac{1}{2} x$, and from the third order term that $f_2(x) = - \frac{1}{3} x$. Hence the modified equation is
\[ \dot{x} = - x - \frac{h}{2} x - \frac{h^2}{3} x - \ldots . \]
We see that the leading order term agrees with the original differential equation. The higher order terms reflect the discretization error.

Since the difference equation $x_{j+1} = x_j - h x_j$ is linear it can be solved exactly and the modified equation doesn't provide any new information. However, the above procedure can be applied just as well to nonlinear difference equations. The only price we pay is that in the nonlinear case the terms tend to become more and more complicated as the order in $h$ increases.

For first order ODEs the notion of a modified equation is well-established \cite[Chapter IX]{hairer2006geometric}. It is easily extended to second order equations \cite{vermeeren2015modified}, which is the relevant setting for this paper.

\begin{defi}
Let $f:\mathbb{R}_+ \times \mathbb{R}^2 \rightarrow \mathbb{R}: (h,x,v) \mapsto F_h(x,v)$ and 
$\Psi:\mathbb{R}_+ \times \mathbb{R}^3 \rightarrow \mathbb{R}: (h,x,y,z) \mapsto \Psi_h(x,y,z)$ be smooth functions. The differential equation $\ddot{x} = F_h(x,\dot{x})$ is a \emph{modified equation} for the second order difference equation $\Psi_h(x_{j-1},x_j,x_{j+1}) = 0$ if for all sufficiently small $h > 0$, every solution $x$ of $\ddot{x} = F_h(x,\dot{x})$ satisfies $\Psi_h \big( x(t-h),x(t),x(t+h) \big) = 0$ for all $t \in \mathbb{R}$.
\end{defi}

In other words, for every solution $x$ of the modified equation, the discrete curve \mbox{$(x(t_0+jh))_{j \in \mathbb{Z}}$} solves the difference equation.
Usually the modified equation is obtained as a formal power series in the step size $h$ that does not converge. To be rigorous in general, the definition should be adapted to handle a right hand side $F_h(x,\dot{x})$ that is a formal power series, but in this paper we will only encounter convergent series.

As an example, consider the differential equation $\ddot{x} = -x$ and its St\"ormer-Verlet discretization, which is obtained by replacing the second derivative of $x$ by the finite difference $\frac{x_{j+1} - 2 x_j + x_{j-1}}{h^2}$. This is a famous example of a \emph{geometric} (i.e.\@ structure-preserving) numerical integrator \cite{hairer2003geometric}. It gives us the difference equation
\begin{equation}\label{sv}
 x_{j+1} - 2 x_j + x_{j-1} = -h^2 x_j. 
\end{equation}
In this particular example the modified equation is unusually simple. It turns out to be sufficient to look for a modified equation of the form
\begin{equation}\label{ansatz}
\ddot{x} = F_h(x) = f_0(x) + h^2 f_2 (x) + h^4 f_4 (x) + \ldots.
\end{equation}
In general we should also include odd order terms and let the $f_i$ depend on $\dot{x}$ as well. It follows from Equation \eqref{ansatz} that
\begin{align}\label{derivatives}
\begin{split}
x^{(3)} &= F_h'\dot{x}, \\
x^{(4)} &= F_h''\dot{x}^2 + F_h'F_h, \\
x^{(5)} &= F_h^{(3)}\dot{x}^3 + 3 F_h''F_h\dot{x} + F_h'^2\dot{x}, \\
x^{(6)} &= F_h^{(4)}\dot{x}^4 + 6 F_h^{(3)}F_h\dot{x}^2 + 5 F_h''F_h'\dot{x}^2 + 3 F_h''F_h^2 + F_h'^2F_h,
\end{split}
\end{align}
where the argument $x$ of $F_h$ has been omitted. We identify $x(t) = x_j$ and $x_{j \pm 1} = x(t \pm h)$. Using Taylor expansion we find
\[
x_{j \pm 1} 
= x \pm h \dot{x} + \frac{h^2}{2} \ddot{x} \pm \frac{h^3}{6} x^{(3)} + \frac{h^4}{24} x^{(4)} \pm \frac{h^5}{120} x^{(5)} + \frac{h^6}{720} x^{(6)} \pm \frac{h^7}{5040} x^{(7)}  + \ldots .
\]
Plugging this into the difference equation \eqref{sv} and replacing derivatives using \eqref{ansatz} and \eqref{derivatives} we find
\begin{align*}
 -h^2 x 
&= h^2 \ddot{x} + \frac{h^4}{12} x^{(4)} + \frac{h^6}{360} x^{(6)} + \ldots \\
&= h^2 (f_0 + h^2 f_2 + h^4 f_4) + \frac{h^4}{12} \left( f_0''\dot{x}^2 + h^2f_2''\dot{x}^2 + f_0'f_0 + h^2 f_0'f_2 + h^2 f_2'f_0 \right) \\
&\qquad + \frac{h^6}{360} \left( f_0^{(4)}\dot{x}^4 + 6 f_0^{(3)}f_0\dot{x}^2 + 5 f_0''f_0'\dot{x}^2 + 3 f_0''f_0^2 + f_0'^2f_0 \right) + \ldots \\
&= h^2 f_0
+ h^4 \left( f_2 + \frac{1}{12} \left( f_0''\dot{x}^2 + f_0'f_0 \right) \right) \\
&\qquad + h^6 \,\bigg( f_4 + \frac{1}{12} \left(f_2''\dot{x}^2 + f_0'f_2 + f_2'f_0 \right) \\
&\hspace{23mm} + \frac{1}{360} \left( f_0^{(4)}\dot{x}^4 + 6 f_0^{(3)}f_0\dot{x}^2 + 5 f_0''f_0'\dot{x}^2 + 3 f_0''f_0^2 + f_0'^2f_0 \right) \bigg) + \ldots .
\end{align*}
The $h^2$-term of this equation gives us $f_0(x)=-x$, hence $f_0' = -1$ and $f_0'' = 0$. The $h^4$-term then reduces to $f_2(x) = -\frac{x}{12}$, hence $f_2' = -\frac{1}{12}$ and $f_2'' = 0$. Finally, the $h^6$-term gives
\[ f_4(x) = - \frac{1}{12}\left(\frac{x}{12}+\frac{x}{12}\right) + \frac{x}{360} = -\frac{x}{90}. \]
Therefore, the modified equation is
\begin{equation}\label{modeqn-truncated}
\ddot{x} = - x - \frac{h^2}{12} x - \frac{h^4}{90} x - \ldots .
\end{equation}

Note that the coefficients $f_j$ are uniquely determined by the above procedure. This shows that if there exists a modified equation that can be represented as a power series 
\[ \ddot{x} = f_0(x,\dot{x}) + h f_1 (x,\dot{x}) + h^2 f_2 (x,\dot{x}) + h^3 f_3 (x,\dot{x}) + \ldots \]
then it is unique. 

\begin{figure}[t]
\includegraphics[width=\linewidth]{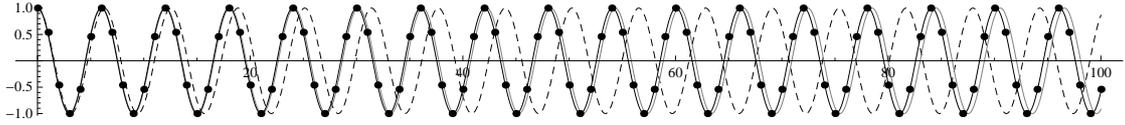}
\caption{The harmonic oscillator. Plotted are an exact solution (dashed curve), a solution of the St\"ormer-Verlet discretization with step size $h=1$ (dots), a solution of the modified equation up to order two (light solid curve), and a solution of the modified equation up to order four (dark solid curve).}\label{fig-harmonic}
\end{figure}

Figure \ref{fig-harmonic} shows a solution of the modified equation truncated after the $h^4$-term. We see that it agrees very well with the discrete flow, despite the large step size. Furthermore, we observe that the St\"ormer-Verlet discretization of the harmonic oscillator with step size $h=1$ is 6-periodic. Since the difference equation has solutions with a period of exactly $6$, so does the modified equation. This suggests that the modified equation is $\ddot{x} = -\frac{\pi^2}{9} x$. 

Because the difference equation \eqref{sv} is linear, this observation is easy to prove. Indeed, we can find an explicit expression for the general solution of the difference equation:
\[
x_j = A e^{- 2 i j \theta} + B e^{2 i j \theta}, 
\]
where $\theta = \arcsin\left(\frac{h}{2}\right)$. Hence we can always find an interpolating curve of the form $x(t) = A e^{- 2 i t \theta/h} + B e^{2 i t \theta/h}$. This curve satisfies the differential equation
\begin{equation}\label{mod-arcsine}
\ddot{x} = - \left(\frac{2}{h} \arcsin\left(\frac{h}{2}\right) \right)^2 x.
\end{equation}
If we set $h=1$ this becomes $\ddot{x} = -\frac{\pi^2}{9} x$.

At this point the connection between modified equations and Equation \eqref{arcsin-series} begins to reveal itself. Our proof of Equation \eqref{arcsin-series} \--- and by corollary of the Basel problem \--- will be based on a comparison of the expressions \eqref{modeqn-truncated} and \eqref{mod-arcsine} of the modified equation. The main difficulty is to find a general expression for the coefficients of the power series in Equation \eqref{modeqn-truncated}.

\section*{Determining the coefficients of (\ref{modeqn-truncated})}

Another consequence of the linearity of the difference equation \eqref{sv} is that the sum $x_{j-k} + x_{j+k}$ can be written as $x_j$ multiplied by a polynomial of degree $2k$ in $h$. More precisely, any solution of \eqref{sv} satisfies
\[
x_{j-k} + x_{j+k} = 2 T_k\!\left(1-\frac{h^2}{2} \right) x_j,
\]
where $T_k$ denotes the $k$-th Chebyshev polynomial of the first kind (proof in the Appendix). This implies that any solution of the modified equation satisfies
\begin{equation}\label{chebyshev}
\begin{split}
x(t-kh)+x(t+kh) &= 2 T_k\!\left(1-\frac{h^2}{2} \right) x(t) \\
&= (-1)^k h^{2k} x(t) + \text{terms of lower order in } h. 
\end{split}
\end{equation}

Now we are in a position to derive an explicit expression for the coefficients of the modified equation \eqref{modeqn-truncated}.\footnote{The argument here is inspired by the derivation in \cite{khan2003mathematical} of suitable coefficients for finite difference methods.}
Fix an arbitrary smooth curve $x$. For every $j$ and $k$ there holds
\[ x(t-jh)-2x(t)+x(t+jh) = (jh)^2 \ddot{x}(t) + \frac{2(jh)^4}{4!} x^{(4)}(t) + \ldots + \frac{2(jh)^{2k}}{(2k)!} x^{(2k)}(t) + \O(h^{2k+2}), \]
or, in matrix form,
\[ \begin{pmatrix}
x(t-h)-2x(t)+x(t+h) \\
x(t-2h)-2x(t)+x(t+2h) \\
\vdots \\
x(t-kh)-2x(t)+x(t+kh)
\end{pmatrix}
=
\begin{pmatrix}
1 & 1 & \ldots & 1 \\
2^2 & 2^4 & \ldots & 2^{2k} \\
\vdots & \vdots & \ddots & \vdots\\
k^2 & k^4 & \ldots & k^{2k} \\
\end{pmatrix}
\begin{pmatrix}
h^2 \ddot{x}(t) \\
\frac{2h^4}{4!} x^{(4)}(t) \\
\vdots \\
\frac{2h^{2k}}{(2k)!} x^{(2k)}(t)
\end{pmatrix} 
 + \O(h^{2k+2}).\]
Using Cramer's rule we solve the above system of linear equations for $h^2 \ddot{x}(t)$,
\[ h^2 \ddot{x}(t) = 
\frac{
\begin{vmatrix}
x(t-h)-2x(t)+x(t+h) & 1 & \ldots & 1 \\
x(t-2h)-2x(t)+x(t+2h) & 2^4 & \ldots & 2^{2k} \\
\vdots & \vdots & \ddots & \vdots\\
x(t-kh)-2x(t)+x(t+kh) & k^4 & \ldots & k^{2k} \\
\end{vmatrix}
}{
\begin{vmatrix}
1 & 1 & \ldots & 1 \\
2^2 & 2^4 & \ldots & 2^{2k} \\
\vdots & \vdots & \ddots & \vdots\\
k^2 & k^4 & \ldots & k^{2k} \\
\end{vmatrix}
}
 + \O(h^{2k+2}).
\]
In the denominator we have a Vandermonde determinant that equals
\[ (k!)^2 \prod_{1 \leq i < j \leq k} (j^2-i^2). \]
The determinant in the numerator is more difficult to evaluate, but if we restrict our attention to the $h^{2k}$-term it becomes just as easy. Assuming the curve $x$ is a solution of the modified equation, we can use Equation \eqref{chebyshev} to find that this term is
\begin{align*}
\begin{vmatrix}
0 & 1 & \ldots & 1 \\
0 & 2^4 & \ldots & 2^{2k} \\
\vdots & \vdots & \ddots & \vdots\\
0 & (k-1)^4 & \ldots & (k-1)^{2k} \\
(-1)^kx(t) & k^4 & \ldots & k^{2k} \\
\end{vmatrix}
&=
-\begin{vmatrix}
1 & \ldots & 1 \\
2^4 & \ldots & 2^{2k} \\
\vdots & \ddots & \vdots\\
(k-1)^4 & \ldots & (k-1)^{2k} \\
\end{vmatrix} x(t) \\
&= - (k-1)!^4 \left( \prod_{1 \leq i < j \leq k-1} (j^2-i^2) \right) x(t).
\end{align*} 
Hence we find that the $h^{2k}$-term of $h^2 \ddot{x}(t)$ equals
\[ - \frac{(k-1)!^4 \prod_{1 \leq i < j \leq k-1} (j^2-i^2)}{(k!)^2 \prod_{1 \leq i < j \leq k} (j^2-i^2)} x(t) , \]
which simplifies to 
\[ -\frac{(k-1)!^2}{k (2k-1)!} x(t). \]
This is the desired explicit expression for the $h^{2k-2}$-term of Equation \eqref{modeqn-truncated}. Hence the modified equation of the difference equation $ x_{j+1} - 2 x_j + x_{j-1} = -h^2 x_j $ is
\begin{equation}\label{mod-series}
\ddot{x} = - \sum_{k=1}^\infty  \frac{2(k-1)!^2}{(2k)!} h^{2k-2} x.
\end{equation}

\section*{Fitting the pieces together}

Equations \eqref{mod-arcsine} and \eqref{mod-series} provide two expressions for the modified equation of the difference equation \eqref{sv}. Since the modified equation, written in the form 
``$ \ddot{x} = \text{power series}$'' is unique, it follows that both expressions coincide:
\[  - \left(\frac{2}{h} \arcsin \frac{h}{2} \right)^2 = - \sum_{k=1}^\infty  \frac{2(k-1)!^2}{(2k)!} h^{2k-2} \]
and hence 
\[\left( \arcsin \frac{h}{2} \right)^2 = \frac{1}{2} \sum_{k=1}^\infty  \frac{(k-1)!^2}{(2k)!} h^{2k} .\]
This expansion is relatively well-known. However, the proofs usually found in the literature\footnote{e.g.\@ the ones sketched in \cite[p. 384-385]{borwein1987pi} and \cite[p. 271]{knopp1954theory}.} are quite complicated. Plugging in $h=1$ we find
\begin{equation}\label{pipi18} 
\sum_{k=1}^\infty \frac{(k-1)!^2}{(2k)!} = \frac{\pi^2}{18} .
\end{equation}
By elementary but nontrivial calculations (presented in the Appendix, following \cite[p. 265-266]{knopp1954theory}) one can show that
\[ \sum_{k=1}^\infty \frac{1}{k^2} = 3 \sum_{k=1}^\infty  \frac{(k-1)!^2}{(2k)!}, \]
which leads to the conclusion that
\[ \sum_{k=1}^\infty \frac{1}{k^2} = \frac{\pi^2}{6}. \]

It is worth mentioning that the closed-form expression \eqref{mod-arcsine} of the modified equation in terms of the arcsine is not needed to arrive at this result. Indeed one can derive Equation \eqref{pipi18} from Equation \eqref{mod-series} and the fact that solutions of the difference equation \eqref{sv} are 6-periodic for $h=1$. The serendipitous observation of this periodicity was the inspiration for our unconventional solution of the Basel problem.


\section*{Conclusion}

Modified equations are an important tool in numerical analysis. One can think of their study as reversing the discretization: one looks for a differential equation for which the discretization would have been exact. This continuous system then provides information about the discrete one.

Usually, modified equations are given by formal power series. When for a specific example the power series does converge, the properties of the discrete system can be useful to evaluate it. This reverses the direction of thinking once more: we are back to using the discrete system to learn about the continuous one. 
However, in this context the discrete system is not a mere approximation of some continuous object, but leads us to an exact evaluation of a power series.

At the moment, the scope of this method is limited to the example presented above. Nevertheless, it is very pleasing to find a connection between a problem that has fascinated mathematicians for centuries and the relatively new concept of modified equations for numerical integrators for ODEs.

\bigskip\noindent \textbf{Acknowledgments.} The author is grateful to the numerous people who gave constructive criticism one some draft of this work, in particular Yuri Suris and the anonymous referees.

The author is supported by the DFG Collaborative Research Center TRR 109 ``Discretization in Geometry and Dynamics''.

\appendix
\section*{Appendix}

\begin{prop}
Let $T_k$ denote the $k$-th Chebyshev polynomial of the first kind. Any solution of the difference equation $ x_{j+1} - 2 x_j + x_{j-1} = -h^2 x_j $ satisfies
\[ x_{j-k} + x_{j+k} = 2 T_k \!\left(1-\frac{h^2}{2} \right) x_j. \]
\end{prop}
\begin{proof}
As observed in the main text, solutions of the difference equation are of the form 
\[ x_j = A e^{- 2 i j \theta} + B e^{2 i j \theta}, \]
where $\theta = \arcsin\left(\frac{h}{2}\right)$. Therefore,
\begin{align*}
x_{j-k} + x_{j+k} &= A e^{- 2 i (j-k) \theta} + B e^{2 i (j-k) \theta} +A e^{- 2 i (j+k) \theta} + B e^{2 i (j+k) \theta} \\
&= A e^{- 2 i j \theta} \left( e^{2 i k \theta} + e^{-2 i k \theta}\right) + B e^{2 i j \theta} \left( e^{-2 i k \theta} + e^{2 i k \theta}\right) \\
&=  2 \cos(2 k \theta) x_j .
\end{align*}
Note that
\[ \cos 2 \theta = 1 - 2 \sin^2 \theta = 1 - 2\left(\frac{h}{2}\right)^2 = 1 - \frac{h^2}{2}, \]
hence, by the trigonometric definition of the Chebyshev polynomials, 
\[ \cos(2 k \theta) = T_k\left(1 - \frac{h^2}{2}\right). \qedhere \]
\end{proof}

\begin{prop}
\[ \sum_{k=1}^\infty \frac{1}{k^2} = 3 \sum_{k=1}^\infty \frac{(k-1)!^2}{(2k)!}. \]
\end{prop}
\begin{proof}
We follow \cite[p. 265-266]{knopp1954theory}. Observe that for any $k$ and $l$
\[ \frac{1}{k^2} - \frac{0!}{k(k+1)} - \frac{1!}{k(k+1)(k+2)} - \ldots - \frac{l!}{k(k+1)\ldots(k+1+l)} = r_{k,l+1}, \]
where 
\[ r_{k,l} = \frac{l!}{k^2(k+1)\ldots(k+l)}. \]
Hence
\begin{align*}
\frac{1}{k^2} &= \sum_{l=1}^k \frac{(l-1)!}{k(k+1)\ldots(k+l)} + r_{k,k} \\
&= \sum_{l=1}^k \frac{(l-1)!}{l} \left( \frac{1}{k\ldots(k+l-1)} - \frac{1}{(k+1)\ldots(k+l)} \right) + r_{k,k}.
\end{align*} 
Therefore
\begin{align*}
\sum_{k=1}^\infty \frac{1}{k^2}
&= \sum_{k=1}^\infty \left( \sum_{l=1}^k \frac{(l-1)!}{l} \left( \frac{1}{k\ldots(k+l-1)} - \frac{1}{(k+1)\ldots(k+l)} \right) + r_{k,k} \right) \\
&= \left( \sum_{l=1}^\infty \sum_{k=l}^\infty \frac{(l-1)!}{l} \left( \frac{1}{k\ldots(k+l-1)} - \frac{1}{(k+1)\ldots(k+l)} \right)\right) + \sum_{k=1}^\infty r_{k,k} \\
&= \sum_{l=1}^\infty \left( \frac{(l-1)!}{l} \frac{1}{l(l+1) \ldots (2l-1)} + r_{l,l} \right) \\
&= \sum_{l=1}^\infty \left( \frac{(l-1)!^2}{l(2l-1)!} + \frac{(l-1)!}{l(l+1)\ldots(2l)} \right) \\
&= \sum_{l=1}^\infty \left( 2 \frac{(l-1)!^2}{(2l)!} + \frac{(l-1)!^2}{(2l)!}  \right) \\
&= 3 \sum_{l=1}^\infty \frac{(l-1)!^2}{(2l)!}. \qedhere
\end{align*} 
\end{proof}

\bibliographystyle{abbrv}
\bibliography{basel-int}
\end{document}